\documentclass[11pt,oneside,reqno]{amsart}

\usepackage[all]{xy}
\usepackage{amsmath}
\usepackage{amsfonts, mathrsfs}
\usepackage{amssymb}
\usepackage{amsthm}
\usepackage{enumerate,enumitem}
\usepackage{xcolor}
\usepackage{fullpage}
\usepackage{graphicx}
\usepackage{url}
\usepackage{hyperref}  
\DeclareMathOperator{\sech}{sech}
 \usepackage{amsaddr}
\usepackage[para,online,flushleft]{threeparttable}
\usepackage{float}
\usepackage{caption}
\usepackage{setspace}
\usepackage{tikz} 
\usepackage{comment}
\usepackage[numbers,sort&compress]{natbib}
\usetikzlibrary{arrows}
\usetikzlibrary{decorations.markings}

\theoremstyle{definition}
\newtheorem{definition}{Definition}[section]
\newtheorem{example}[definition]{Example}
\newtheorem{remark}[definition]{Remark}

\theoremstyle{plain}
\newtheorem{theorem}[definition]{Theorem}
\newtheorem{lemma}[definition]{Lemma}

\title{Convergence of Solutions of the BBM and BBM-KP model Equations
}

\author{Jacob B. Aguilar, Ph.D.}
\address{Department of Mathematics, Saint Leo University, Saint Leo, FL 33574, \\ Email address: jacob.aguilar@saintleo.edu }

\author{Michael M. Tom, Ph.D.}
\address{Department of Mathematics, 
Louisiana State University, Baton Rouge, LA 70803, \\ Email address: tom@math.lsu.edu}


\begin{document}
\maketitle

\begin{abstract}
The Benjamin-Bona-Mahony (BBM) equation has proven to be a good approximation for the unidirectional propagation of small amplitude long waves
in a channel where the crosswise variation can be safely ignored. The Benjamin-Bona-Mahony-Kadomtsev-Petviashvili (BBM-KP) equation is the regularized version of the  Kadomtsev-Petviashvili equation which arises in various modeling scenarios corresponding to nonlinear dispersive waves that propagate
principally along the $x$-axis with weak dispersive effects undergone in the direction parallel to the $y$-axis and normal to the primary direction of propagation. There is much literature on mathematical studies regarding these well known equations, however the relationship between the solutions of their underlying pure initial value problems is not fully understood. In this work, it is shown that the solution of the Cauchy problem for the BBM-KP equation converges
to the solution of the Cauchy problem for the BBM equation in a suitable function space, provided that
the initial data for both equations are close as the transverse variable $y \rightarrow \pm \infty$.
\end{abstract}

% ===============================================
\section{Introduction and Background}\label{sec:int}
The pure initial value problem for the Benjamin-Bona-Mahony (BBM) equation
%----------------------------
\begin{equation}
\begin{cases}
u_t +uu_x -u_{xxt}= 0
\hspace{10pt}\text{for}\hspace{10pt} (x,t) \in \mathbb{R} \times \mathbb{R}_+,\\ 
u(x,0)=u_0(x),\\
\end{cases}
    \label{BBM}
\end{equation}
%----------------------------
has been studied by various authors. The BBM equation, appearing in (\ref{BBM}) was introduced in \cite{benjamin1972model} as the regularized counterpart of the well-known Korteweg-de-Vries (KdV) equation  \cite{bona1981evaluation,hammack1974korteweg,zabusky1971shallow}. This equation was originally proposed as a model for one-dimensional, unidirectional small-amplitude long-waves on inviscid fluids. In the setting of shallow-water waves, $u=u(x,t)$ represents the displacement of the water surface or velocity at time $t$ and location $x$. In addition to modeling long weakly dispersive surface waves in liquids, the BBM equation 
has been utilized as a model for hydromagnetic waves in cold plasma, acoustic waves in anharmonic crystals and acoustic-gravity waves in compressible fluids 
\citep[p.~612]{wazwaz2010partial}. 

One can obtain the BBM equation from the KdV equation by observing that under suitable conditions $u_x \approx -u_t$. This derivative approximation permits the  replacement of the third-order term $u_{xxx}$ by the mixed term  $-u_{xxt}$ and results in the following bounded dispersion relation
\begin{equation*}
\omega_1(\xi)=\frac{\xi}{1+\xi^2}. \qquad \text{(Dispersion relation of BBM equation)}
\label{Dispersion relation of BBM}
\end{equation*}
The dispersion relation relates the time evolution of a system to its spatial structure. This function uniquely characterizes the linear part of a system and encodes information regarding the propagation of its corresponding traveling wave solutions. Upon comparing the dispersion relation of the BBM equation to that of the KdV equation, it is immediate that dispersion relation of the BBM equation is preferable. Indeed, the function $\omega_1$ is bounded and does not exhibit bad limiting behavior as the unbounded dispersion relation of the KdV equation. Both equations admit solitary wave solutions \cite{zeng2003existence,sirendaoreji2004new,wang2006exact}. However, in a physical context, the unbounded dispersion relation of the KdV equation results in wave solutions which are allowed to  propagate with infinite speed. As pointed out in \cite{benjamin1972model}, this is one of many deficiencies of the KdV equation with respect to modeling scenarios. Consequently, the BBM equation serves as a more physically accurate model and has attracted much attention from researchers.

The derivation of the BBM equation hinges on the implementation of a Boussinesq scaling regime. In such  parameter regimes one specializes to the case of long waves which possess a small amplitude in comparison with the depth of the water. Subsequently, a sequence of approximate Hamiltonians can be constructed by expanding the Dirichlet-Neumann operator and retaining terms up to a specified order. The governing idea behind such small-amplitude long wave derivation schemes resides in approximating a Hamiltonian evolutionary system by fixing its phase space and Poison structure and replacing the Hamiltonian density function by an appropriate approximation (see \cite{benjamin1984impulse,craig1994hamiltonian} for a detailed review of the Hamiltonian perturbation theory utilized in such variational derivations). 

A restriction concerning the application of the BBM equation as a practical model for water waves is that the BBM equation is strictly one-dimensional, i.e. one spatial dimension plus time, whereas the surface of a water wave is two-dimensional. In an effort to remedy this, the following
pure initial value problem for the Benjamin-Bona-Mahony-Kadomtsev-Petviashvili (BBM-KP) equation was proposed 
\begin{equation}
\begin{cases}
(\eta_{t} + \eta_{x} + \eta \eta_{x} - \eta_{xxt})_x + \gamma \eta_{yy} = 0 \hspace{10pt}\text{for}\hspace{10pt} (x,y,t) \in \mathbb{R}^2 \times \mathbb{R}_+\hspace{10pt}\text{and}\hspace{10pt}\gamma = \pm1,\\  
\eta(x,y,0)=\eta_0(x,y).
\end{cases}
\label{BBM-KP}
\end{equation}

The BBM-KP equation, featured in (\ref{BBM-KP}), is the regularized version of the usual KP equation introduced in \cite{kadomtsev1970stability}. This model arises in various contexts where nonlinear dispersive waves propagate principally along the $x$-axis with weak dispersive effects undergone in the direction parallel to the $y$-axis and normal to the main direction of propagation. A diversity of exact travelling wave solutions of the BBM-KP equation and its generalizations have been formally derived, including: solitons, compactons, solitary patterns and periodic solutions  \cite{de1997solitary,wazwaz2005exact,wazwaz2008extended,song2010exact,alam2013exact,ouyang2014traveling}.

The linearized dispersion relation of the BBM-KP equation is 
\begin{equation*}
\omega_2(\xi,\mu) = \frac{\xi^2 + \gamma \mu^2}{\xi(1+\xi^2)}, \quad \gamma = \pm1. \qquad \text{(Dispersion relation of BBM-KP equation)}
\label{Dispersion relation of BBMKP}
\end{equation*}
The dispersion relation $\omega_2$ is a good approximation to the dispersion relation of the usual KP equation, but does not possess the unwanted limiting behavior as $\xi \rightarrow +\infty$ \cite{bona2002cauchy}.

To procure the BBM-KP equation from the BBM equation, the horizontal coordinates are oriented in a manner so that the $x$-direction is the principal direction of wave propagation. Additionally, it is assumed that wave amplitudes are relatively small, the water is shallow typical to horizontal wavelengths, and the waves are nearly one dimensional. The resulting equation is the BBM formulated in the KP sense, and is naturally referred to as the BBM-KP equation. This equation models small amplitude long waves in $(2+1)$ space, principally moving along the $x$-axis. Similar to the KP equations, (\ref{BBM-KP}) is called the BBM-KP I or BBM-KP II, depending if $\gamma$ is equal to negative or positive unity, respectively. In a physical context, the sign of the parameter $\gamma$ corresponds to whether the surface tension is neglected or not. As in the case of the BBM equation, the BBM-KP equation can be formally derived by utilizing the unifying framework covered by Craig and Groves \cite{craig1994hamiltonian}. Specifically, the usual KP equations are derived by separating out the right and left-going waves and
employing the preservation of the Hamiltonian structure  under changes of the dependent and independent variables. Afterwards, the BBM-KP equation is obtained through applying the derivative approximation used in the derivation of the BBM equation.

Despite the vast amount of literature regarding the pure initial value problems (\ref{BBM}) and (\ref{BBM-KP}), the qualitative relationship between their solutions is not fully understood. Our work is a step towards in this direction.  Particularly, we show that the solutions of the Cauchy problems associated to the BBM and BBM-KP model equations converge in the $L^2$-based Sobolev Class $H_x^k(\mathbb{R}) \subset L^2(\mathbb{R})$ for all $k \geq 1$, provided their corresponding initial data are close in $H_x^k(\mathbb{R})$ as the transverse variable $y \rightarrow \pm \infty$. Specifically, we have the following theorem.
\begin{theorem}
Let $\eta$ be the solution of the Cauchy problem (\ref{BBM-KP}) with initial data $\eta(x,y,0)=\psi(x,y)$, where $\psi \in H^s_{-1}(\mathbb{R}^2)$. Assume that $\phi^+$ and $\phi^-$ are two functions in $H_x^k(\mathbb{R})$ such that 
$$\lim_{y \rightarrow \pm \infty}||\psi(\cdot,y)-\phi^\pm(\cdot)||_{H^k_x(\mathbb{R})}=0.$$
If $u^+$ and $u^-$ are nontrivial solutions of the Cauchy problem (\ref{BBM}) corresponding to initial data $u^\pm(x,0)=\phi^\pm(x)$, i.e. $u^+$ and $u^- \neq 0$ a.e., then 
$$\lim_{y \rightarrow \pm \infty}||\eta -u^\pm||_{H^k_x(\mathbb{R})}=0,$$
for all $k\geq 1$ and $s \geq k+1$.
\label{thm:jacobtom}
\end{theorem}
In the statement of the above theorem, the plus (minus) superscript over the initial data denotes whether the transverse variable $y$ of $\psi$ is approaching positive or negative infinity and corresponds to the plus (minus) superscript over the solution (see Definition \ref{def:data} in Section \ref{sec:proof}).

This manuscript provides a deeper qualitative understanding concerning the intricate relationship between the solutions of the BBM (\ref{BBM}) and BBM-KP (\ref{BBM-KP}) pure inital value problems and is organized as follows: Section \ref{sec:notate} features the mathematical notation and  framework which is utilized in the rigorous calculations to follow,  Section \ref{sec:exist} summarizes the existence theory of the pure initial value problems being studied, Section \ref{sec:relate}
is dedicated to the mathematical relationship between the BBM and BBM-KP model equations, and Section \ref{sec:proof} contains the preliminary lemmas and proof of the main results.
% ===============================================
 % ===============================================
\section{Notation and Framework}\label{sec:notate}
In this section, an overview of the notation and mathematical framework required for obtaining our result is introduced. For an in-depth look into the function classes utilized in this work, and many other closely related topics in the fields of functional analysis and nonlinear partial differential equations, the reader is referred to \cite{tao2006nonlinear,grafakos2008classical,evans2010partial}. 

Firstly, let $\mathbb{R}_+:=\{x \in \mathbb{R}:x > 0\}$ and $\mathbb{R}_0:=\mathbb{R}_+ \cup \{0\}$ denote the sets of strictly positive and non-negative real numbers, respectively. Additionally, the ordered pair $(\xi, \mu)$ corresponds to the Fourier variables dual to $(x, y)$. All integrals will be with respect to Lebesgue measure $\lambda$ for any complex-valued measurable function $\varphi$. Use is made of the Banach space $L^\infty(\mathbb{R})$, i.e. the space of all real valued measurable functions  which are essentially bounded, characterized by the following norm
%----------------------------
\begin{equation*}
	| \varphi|_\infty \equiv \inf \left\lbrace C \geq 0 : \lambda(\{ x: |\varphi |> C \})=0 \right\rbrace. 
\end{equation*}
A smooth function belongs to the Schwartz class $S(\mathbb{R}^n)$ provided that all of its derivatives are rapidly decreasing. It is well know that $S(\mathbb{R}^n)$ is a Frechet Space and, as a result, has a dual denoted $S'(\mathbb{R}^n)$, i.e. the space of tempered distributions.
If $X$ and $Y$ are both Banach spaces, we denote a continuous embedding between them by $\hookrightarrow$, i.e. $X \hookrightarrow Y$ means that $||u||_Y \lesssim ||u||_X$. When attempting to obtain a contraction mapping for the intergral formulation of a given Cauchy problem, one needs to first localize in time. Consider the following function space $C\left([a,b];X\right)$ equipped with the norm
$$||u||_{C^{a}_{b}X} := \sup_{t \in [a,b]} ||u||_X.$$ It directly follows that if $X$ is a Banach space, then so is  $C\left([a,b];X\right)$ \cite{fabec2014non}. Next, we proceed to define the necessary function spaces and their associated norms. Let $H^s(\mathbb{R}^2)$ denote the classical Sobolev space equipped with the norm
$$||\eta||_{H^s(\mathbb{R}^2)} = \Big(\int_{\mathbb{R}^2}(1 + \mu^2 + \xi^2)^s|\hat{\eta}(\xi,\mu)|^2 d\xi d \mu \Big)^\frac{1}{2}.$$
Analogously, $H^k_x(\mathbb{R})$ will denote the Sobolev space in just the spatial variable $x$, endowed with the following norm
$$||f||_{H^k_x(\mathbb{R})} = \Big(\int_{\mathbb{R}}(1 + \xi^2)^k|\hat{f}(\xi)|^2 d \xi \Big)^\frac{1}{2}.$$
Extensive use is made of the space 
$H^s_{-1}(\mathbb{R}^2)=\{\eta \in S'(\mathbb{R}^2):||\eta||_{H^s_{-1}(\mathbb{R}^2)} < \infty\}$, 
supplied with the norm
$$||\eta||_{H^s_{-1}(\mathbb{R}^2)} = \Big(\int_{\mathbb{R}^2}(1+|\xi|^{-1})^2(1 + \xi^2 + \mu^2)^s|\hat{\eta}(\xi,\mu)|^2 d\xi d \mu \Big)^\frac{1}{2}.$$

The pseudo-differential operator $\partial^{-1}_x f$ is defined via the Fourier transform as 
$$\widehat{\partial^{-1}_x f} := \frac{1}{i\xi} \hat{f}(\xi,y).$$
Due to the singularity of the symbol $\xi^{-1}$ at $\xi =0$, one requires that $\hat{f}(0,y)=0$
(the Fourier transform in the variable $x$), which is clearly equivalent to
$\int_{\mathbb{R}}f(x,y) dx=0.$
In what follows, $\partial^{-1}_x f\in L^2(\mathbb{R}^2)$ means there exists an $L^2(\mathbb{R}^2)$ function $g$ such 
that $g_x=f$, at least in the distributional sense. When we write $\partial^{-k}_x \partial^{m}_y$ for $\left( k,m\right) \in \mathbb{Z}^+ \times \mathbb{Z}^+$, we implicitly assume that the operator is well-defined. As covered in \cite{molinet2007remarks}, this imposes a constraint on the solution $u$. This implies that $u$ is an $x$ derivative of a suitable parent function. Such a condition can be accomplished provided that $u\in S'(\mathbb{R}^2)$ is such that $\xi^{-k} \mu^{m} \hat{u}(\xi,\mu, t)\in S'(\mathbb{R}^2)$, or if $u(x,y,t)= \frac{\partial}{\partial_x}v(x,y,t)$ for $v \in C^1_x(\mathbb{R})$, i.e. the space of continuous functions possessing a continuous derivative with respect to $x$.

It is worth pointing out that the second possibility mentioned above imposes a decay condition on the function $u$. Precisely, for fixed $y$ and $t \in \mathbb{R}_+$ it is required that $u \rightarrow 0$ as $x \rightarrow \pm \infty$. For all $(t,y) \in \mathbb{R}_+ \times \mathbb{R}$, this setup results in
$\int_{\mathbb{R}}u(x,y,t) dx=0.$ Formally, we localize in time and define the class of functions 
$\mathcal{X}^s(\mathbb{R}^2):=\{u : u \in H^s(\mathbb{R}^2) \cap H^{s-1}(\mathbb{R}^2)\}$, 
where $v(x,y,t)= {\partial^{-1}_x}u(x,y,t)$ and $w(x,y,t) = \int^{x'}_{-\infty}u(x',y,t) dx'$. Thus,
$\hat{u}(t)= \hat{w}(t)$ in  $S'(\mathbb{R}^2)$ for all $t \in \left[-T,T\right]$
and as a result $v=w$ due to that fact that the Fourier transform is an isomorphism on $S'(\mathbb{R}^2)$. Since $u \in C\left([-T,T]; \mathcal{X}^s(\mathbb{R}^2)\right)$ and $s>2$, it is a direct consequence that $u \in L^1(\mathbb{R}^2)$ \cite{mammeri2009unique}. Moreover, $\partial^{-1}_x u \in C\left([-T,T]; H^s(\mathbb{R}^2)\right)$ and the fact that $S(\mathbb{R}^n)$ is dense in $H^s(\mathbb{R}^2)$ implies that $v \rightarrow 0$ as $x \rightarrow \pm \infty$. Therefore, it transpires that 
\begin{displaymath}
\int_{\mathbb{R}}u(x',y,t) dx' = \lim_{x \rightarrow \infty}\int^{x'}_{-\infty}u(x',y,t) dx':=\lim_{x \rightarrow \infty} w(x,y,t)=\lim_{x \rightarrow \infty} v(x,y,t)= \lim_{x \rightarrow \infty} {\partial^{-1}_x}u(x,y,t)=0.
\end{displaymath}
As a result, an additional requirement is imposed on the initial data to a given Cauchy problem that contains the operator $\partial^{-k}_x \partial^{m}_y$. In response, one is often forced to turn to the more regular, weighted anisotropic Sobolev spaces in order to obtain a contraction mapping to the Duhamel formulation of the corresponding initial value problem.

Furthermore, we adopt the following convention in order to simplify various calculations and estimates involved in the proofs. 
\begin{definition}\label{def:bound}
Let $A,B \in \mathbb{R}$, we denote $A \vee B := \max\{A,B\}$ and $A \wedge B := \min\{A,B\}$. Furthermore, the notation $A \lesssim B$ (respectively $A \gtrsim B$) means that there exists an absolute positive constant $C$ such that $A \leq C B$  (respectively $A \geq C B$).
\end{definition}
% ===============================================
\section{Existence theory}\label{sec:exist}
This section contains a summarization of the existence theory for the Cauchy problems (\ref{BBM}) and (\ref{BBM-KP}). The question of global well-posedness of the BBM pure initial value problem (\ref{BBM}) was recently answered in \cite{bona2009sharp}. Bona and Tzvetkov showed that the Cauchy problem (\ref{BBM}) is globally well-posed in the $L^2$-based Sobolev class $H^s(\mathbb{R})$, for $s \geq 0$. More precisely, they proved the following result \cite{bona2009sharp}. 
\begin{theorem}  Fix $s \geq 0$. For any $u_0 \in H^s(\mathbb{R})$, there exists a $T=T(||u_0||_{H^s}) > 0$ and a unique solution
$u \in C([0,T]; H^s(\mathbb{R}))$ of the initial value problem (\ref{BBM}). 
Moreover, for $R>0$, let $B_R$ connote a ball of radius $R$ centered at the origin in $H^s(\mathbb{R})$ and let $T=T(R) >0$ denote a uniform existence time for the initial value problem (\ref{BBM}) with $u_0 \in B_R$. Then the correspondence $u_0 \mapsto u$ which associates to $u_0$ the solution $u$ of the initial value problem (\ref{BBM}) with initial data $u_0$ is a real analytic mapping of $B_R$ to $C([T,-T]; H^s(\mathbb{R}))$.
\label{BBM GWP}
\end{theorem}
The above theorem improved earlier known results proven by Benjamin et al. \cite{benjamin1972model}, where the initial value problem (\ref{BBM}) was shown to be globally well-posed for data in $H^k$, where $k\in \mathbb{Z}$ such that $k \geq 1$. It should also be noted that the authors in \cite{bona2009sharp} proved that  the initial value problem (\ref{BBM}) is ill-posed for data in $H^s(\mathbb{R})$ such that $s<0$. Indeed, the flow map $u_0 \mapsto u(t)$ is not even $C^2$. Precisely, the ensuing result was proven in \cite{bona2009sharp}. 
\begin{theorem}  For any $T>0$ and $s<0$, the flow-map $u_0 \mapsto u(t)$ associated to the Cauchy problem (\ref{BBM}) is not of class $C^2$ from $H^s$ to $C([0,T]; H^s(\mathbb{R}))$.
\end{theorem}
%%%
Regarding the well-posedness of the BBM-KP pure initial value problem,
Bona et al. \cite{bona2002cauchy} showed that the Cauchy problem (\ref{BBM-KP}) can be solved by Picard Iteration yielding to local and global well-posedness results. In particular, the flow map was shown to be smooth and the well-posedness results were established for a class of equations involving pure power nonlinearities and general dispersion in $x$. Specifically, the following theorem was proven in \cite{bona2002cauchy}. 

\begin{theorem}
 Let $\eta_0 \in H^s_{-1}(\mathbb{R}^2)$ with $s > \frac{3}{2}$. Then there exist a $T_0$ such that the initial-value problem 
(\ref{BBM-KP}) has a unique solution $\eta \in C([0,T]; H^s_{-1}(\mathbb{R}^2))$, $\partial^{-1}_x\eta_y \in C([0,T]; H^{s-1}_{-1}(\mathbb{R}^2))$,
with $\eta_t \in C([0,T]; H^{s-2}(\mathbb{R}^2))$. Moreover, the map $\psi \rightarrow \eta$ is continuous from  $H^s_{-1}(\mathbb{R}^2)$
to $C([0,T_0]; H^s_{-1}(\mathbb{R}^2)).$
\label{BBM-KP GWP}
\end{theorem}

When applied to the usual BBM-KP Cauchy problem (\ref{BBM-KP}), the results established in \cite{bona2002cauchy} imply global well-posedness, regardless of the sign of $\gamma$, in the following space

$$W_1(\mathbb{R}^2)=\{\eta_0 \in L^2(\mathbb{R}^2):  ||\eta_0||_{L^2} + ||\partial_x \eta_0||_{L^2}  + ||\partial_{xx} \eta_0||_{L^2}+ ||\partial^{-1}_x \partial_y \eta_0||_{L^2}+ ||\partial_y \eta_0||_{L^2} < \infty \}.$$

Afterwards, Saut and Tzvetkov \cite{saut2004global} improved these global well-posedness results to the following energy space
$Y=\{\eta_0 \in L^2(\mathbb{R}^2): \partial_x\eta_0 \in L^2(\mathbb{R}^2)\}.$
% ===============================================
 % ===============================================
\section{Relationship Between the Model Equations}\label{sec:relate}
% ===============================================
The primary aim of this section is to highlight the mathematical relationship between the BBM and BBM-KP model equations. Section \ref{sec:int} sheds light on the fact that the derivation of the BBM-KP equation hinges on similar physical assumptions as those utilized for the BBM equation, with the additional condition that the wave motion simultaneously experiences weak variation along the transverse direction. 

In order to formalize the matter at hand, we recast an argument used by Molinet et al. \cite{molinet2007remarks} into the theme of BBM type equations. To this end we start with a one dimensional long-wave dispersive equation of the BBM type, i.e.
\begin{equation}
u_{t} + \alpha u_{x} + u u_{x} - L u_{t}= 0, \quad u = u(x,t) \quad \text{for} \quad (\alpha,x,t) \in \mathbb{R}^2 \times \mathbb{R}_0.\\
\label{4.1}
\end{equation}
The operator $L$ appearing in equation (\ref{4.1}) above is a Fourier multiplier defined as
\begin{equation*}
\widehat{L \varphi(\xi)} = m(\xi) \hat{\varphi(\xi)}.
\end{equation*}
For a real function $m$, the circumflex denotes the function’s Fourier transform and the symbol $m$ associated to the operator $L$ is assumed to be homogeneous. When $\alpha=0$, the case  $m(\xi) = {\xi}^2$ corresponds to the operator $L=  {\partial}^2_x$ and produces the BBM equation, appearing in (\ref{BBM}). In the setting of water wave modeling scenarios, the multiplier $m(\xi)$ stands for the phase velocity and its sign depends on the surface tension parameter, as previously discussed for the BBM-KP equation in Section \ref{sec:int}. 

As mentioned in \cite{kadomtsev1970stability}, the correction to equation (\ref{4.1}) due to weak transverse effects is independent of the dispersion in $x$ and is solely related to the finite propagation speed properties of the linear transport operator $M = \partial_t + \partial_x$. We recall that $M$ gives rise to right moving unidirectional waves with unit speed, i.e. an initial wave profile $u_0(x)$ evolves under the flow of $M$ as $u_0(x-t)$. Accordingly, we define a weak transverse perturbation of $u_0(x)$ to be a two dimensional function $\eta_0(x,y)$ close to $u_0(x)$ when localized in the frequency region $\Big| \frac{\mu}{\xi} \Big|\ll 1$.

 Let $m(\partial_x,\partial_y)$ be the Fourier multiplier with real symbol $m(\xi, \mu)$. The governing idea is to seek a perturbation $\tilde{M} = \partial_t + \partial_x + m(\partial_x,\partial_y)$ of $M$ such that the wave profile of $\eta_0(x,y)$ undergoes negligible variation when evolving under the flow of $\tilde{M}$. As pointed out in \cite{molinet2007remarks}, a natural generalization of the flow of $M$ to $\mathbb{R}^2$ is the flow of the wave operator $W = \partial_t + \sqrt{-\Delta}$, which also exhibits the finite propagation speed property. Provided that $\Big| \frac{\mu}{\xi} \Big|\ll 1$, the approximation $\xi + \frac{1}{2} \xi^{-1} \mu^2 \approx \pm \sqrt{\xi^2 + \mu^2}$ holds and it is a direct consequence that
$$ \partial_t + \partial_x + \frac{1}{2}{\partial}^{-1}_x{\partial}^2_y \sim \partial_t + \sqrt{-\Delta}.$$
The above operator approximation leads to the correction $m(\partial_x,\partial_y)= \frac{1}{2}{\partial}^{-1}_x{\partial}^2_y$. Now, if we weight the multiplier $m$ by a constant, e.g. consider $\beta m(\partial_x,\partial_y)$ with real symbol $m(\xi, \mu)$ and constant $\beta \in \mathbb{R}-\{0\}$, then we obtain the correction $m(\partial_x,\partial_y)= \frac{1}{2\beta}{\partial}^{-1}_x{\partial}^2_y$. Therefore, we arrive at the following two dimensional model
\begin{equation*}
u_{t} + \alpha u_{x} + u u_{x} - L u_{t} + \frac{1}{2\beta}{\partial}^{-1}_x{\partial}^2_yu= 0. \\
\label{3.2}
\end{equation*}
Taking $\alpha=1$, $\beta=(2\gamma)^{-1}$ and
$L=  {\partial}^2_x$, we obtain the BBM-KP equation, featured in (\ref{BBM-KP}). This mathematical argument is in agreement with the fact that the BBM-KP equation models weakly dispersive long waves which essentially propagate in one direction with weak transverse effects. Indeed, the brief covering of the model formulation in Section \ref{sec:int} employed the assumption that the average wave length in the $x$ direction is much larger than the average wave length in the $y$ direction. In this vein, it is natural to view the BBM-KP equation as a weak transverse perturbation of the BBM equation.

This discussion strongly suggests an intricate mathematical relationship between the solutions of the Cauchy problems associated to the BBM and BBM-KP model equations. The results established in Theorem \ref{thm:jacobtom} further contribute to the current knowledge concerning this relationship and are proven in Section \ref{sec:proof} below.
% ===============================================
  
%%%%%%%%%%%%%%%%%%%%%%%%%%%%%%%%%%%%%%%%%%%%%

%%%%%%%%%%%%%%%%%%%%%%%%%%%%%%%%%%%%%%%%%%%%%%%%%%%%%%
\section{
Proof of the main results
}\label{sec:proof}
This section contains the proof of our main result, namely Theorem \ref{thm:jacobtom}. We present the preliminary definitions and proceed to prove the Theorem. The required lemmas necessary for obtaining our result are proved in passage. 
%%%
\begin{definition}\label{def:data}
Assume that $\psi \in H^s_{-1}(\mathbb{R}^2)$ with $s > \frac{3}{2}$. We let $u^+$ and $u^-$ denote the solutions to the pure initial-value problem (\ref{BBM}), emanating from Theorem \ref{BBM GWP}, corresponding to initial data $\phi^+$ 
and $\phi^-$ defined by the equations below:
\begin{equation} 
	  \phi^+(\cdot):=\lim_{y \to +\infty} \psi(\cdot,y) \quad \text{and} \quad  
      	\phi^-(\cdot) :=\lim_{y \to -\infty} \psi(\cdot,y).   
	\label{intdatadef}
    \end{equation}
\end{definition}
\begin{example}\label{examp} Concerning an example of such initial data described in Definition \ref{def:data}, one could let $\psi(x,y) = \phi(x) + \sech y$ for $\phi \in H^s(\mathbb{R})$.
\end{example}

In an effort to simplify some of the calculations involved in the proof of Theorem \ref{thm:jacobtom} below, we employ a slight abuse of notation. Particularly we introduce the notation $C_1^\pm$, where the superscript corresponds to whether the transverse variable $y$ is approaching positive or negative infinity. A similar naming convention was deployed in the form of a superscript on the functions $u^\pm$ and their corresponding initial data $\phi^\pm$.  

%%%%%%%%%%%%%%%%%%%%%%%%%%%%%%%%%%%%%%%%%%%%%%%%%%%%
\subsection{Proof of Theorem \ref{thm:jacobtom}} \label{sec:ProofHor}
Now that we have established all of the necessary ingredients, we advance to prove the main result.
\begin{proof}
Suppose $u^+$ and $u^-$ are solutions to the Cauchy problem (\ref{BBM}), emerging from Theorem \ref{BBM GWP}, corresponding to initial data
$\phi^+$ and $\phi^-$ as described in Definition \ref{def:data} above. Moreover, let $\eta$ be the solution to the Cauchy problem (\ref{BBM-KP}) arising from Theorem \ref{BBM-KP GWP} with initial data $\psi$. Define the following function 
\begin{equation}
w(x,y,t) = \eta - \frac{1}{2} [u^{+} + u^{-}] - \frac{1}{2}[u^{+} - u^{-}] \tanh y.
\label{wfunction}    
\end{equation}
Observe that 
\begin{equation*}
w(x,y,0) = \psi(x,y) - \frac{1}{2} [\phi^{+}(x) + \phi^{-}(x)] - \frac{1}{2}[\phi^{+}(x) - \phi^{-}(x)] \tanh y
\label{wfunction0}    
\end{equation*}
and hence
$$\lim_{y \to \pm \infty} w(x,y,0)=0.$$ 
A straightforward calculation shows that $w$ given by Equation \ref{wfunction}  above solves the following initial value problem
\begin{equation}
 \left\{
\begin{array}{l}
\displaystyle w_t + w_x -w_{xxt} - \partial^{-1}_x \eta_{yy} + w w_x + \frac{1}{2}(1 + \tanh y) (u^+ w)_x + \frac{1}{2}(1- \tanh y)(u^- w)_x \\
\displaystyle - \frac{1}{4} ( 1 - \tanh^2 y) ( u^+ u^+_x + u^- u^-_x - u^+u^-_x - u^-u^+_x) = 0, \\
\displaystyle w(x,y,0)= \psi(x,y) - \frac{1}{2}[ \phi^+ +\phi^-] - \frac{1}{2}[\phi^+ -\phi^-] \tanh y. \\
    \end{array} \right.
    \label{wIVP}
\end{equation}
We now venture into the task of estimating $||w||_{H^1_x(\mathbb{R})}$ for any $y \in \mathbb{R}$. To this end, we multiply Equation \ref{wIVP} by $w$ and integrate over $\mathbb{R}$, in the spatial variable $x$, 
to obtain the following integral
$$\int_{\mathbb{R}} [w w_t + w w_x - w w_{xxt} -  w \partial^{-1}_x \eta_{yy} + w^2 w_x + w \frac{1}{2}(1 + \tanh y)  (u^+ w)_x +w \frac{1}{2}(1- \tanh y)(u^- w)_x $$
$$-w \frac{1}{4} ( 1 - \tanh^2 y) ( u^+ u^+_x + u^- u^-_x - u^+u^-_x - u^-u^+_x)]\,dx =0.$$
After a few integration by parts, we arrive at the following estimate
$$\frac{1}{2} \frac{d}{dt}  \Big[\int_{\mathbb{R}}w^2 dx + \int_{\mathbb{R}}w^2_x dx\Big] \leq \Big|\int_{\mathbb{R}} w \partial^{-1}_x \eta_{yy} dx \Big | + \frac{1}{2}\Big | \int_{\mathbb{R}} (1 + \tanh y) u^+w_xw dx \Big |+ $$    
$$ \frac{1}{2}\Big |\int_{\mathbb{R}}(1 - \tanh y) u^- w_xw  dx \Big |+ \frac{1}{4}\Big| \int_{\mathbb{R}} ( 1 - \tanh^2 y)w  u^+ u^+_x\,dx\Big|+ \frac{1}{4}\Big| \int_{\mathbb{R}} ( 1 - \tanh^2 y)w  u^- u^-_x dx \Big|$$        
$$ +\frac{1}{4} \Big | \int_{\mathbb{R}}  ( 1 - \tanh^2 y) w  u^+ u^-_x dx \Big | + \frac{1}{4} \Big |\int_{\mathbb{R}} ( 1 - \tanh^2 y)w  u^- u^+_x dx \Big |.$$
Making use of H$\ddot{o}$lders inequality, it follows that 
$$\frac{1}{2} \frac{d}{dt} || w ||^2_{H^1_x(\mathbb{R})} \leq ||w||_{L^2(\mathbb{R})}  ||\partial^{-1}_x \eta_{yy}||_{L^2(\mathbb{R})} + \frac{(1 + \tanh y)}{2}| u^+|_\infty||w||_{L^2(\mathbb{R})}||w_x||_{L^2(\mathbb{R})}$$  
$$ + \frac{(1 - \tanh y)}{2} | u^-|_\infty||w||_{L^2(\mathbb{R})}||w_x||_{L^2(\mathbb{R})}+ \frac{( 1 - \tanh^2 y)}{4}  \text{ } | u^+|_\infty||w||_{L^2(\mathbb{R})}||u^+_x||_{L^2(\mathbb{R})}$$        
$$ + \frac{( 1 - \tanh^2 y)}{4}  \text{ }  | u^-|_\infty||w||_{L^2(\mathbb{R})}||u^-_x||_{L^2(\mathbb{R})}+ \frac{( 1 - \tanh^2 y)}{4}  \text{ }  | u^+|_\infty||w||_{L^2(\mathbb{R})}||u^-_x||_{L^2(\mathbb{R})}$$   
$$ + \frac{( 1 - \tanh^2 y)}{4}  \text{ }  | u^-|_\infty||w||_{L^2(\mathbb{R})}||u^+_x||_{L^2(\mathbb{R})}.$$
An application of Young's inequality yields
$$\frac{1}{2} \frac{d}{dt} || w ||^2_{H^1_x(\mathbb{R})} \leq ||w||_{H^1_x(\mathbb{R})}  ||\partial^{-1}_x \eta_{yy}||_{L^2(\mathbb{R})} +  \frac{(1 + \tanh y)}{2}| u^+|_\infty||w||^2_{H^1_x(\mathbb{R})} $$  
$$ +  \frac{(1 - \tanh y)}{2}| u^-|_\infty||w||^2_{H^1_x(\mathbb{R})} + \frac{( 1 - \tanh^2 y)}{4}  \text{ } | u^+|_\infty||w||_{H^1_x(\mathbb{R})}||u^+_x||_{L^2(\mathbb{R})}$$       
$$ + \frac{( 1 - \tanh^2 y)}{4} \text{ }  | u^-|_\infty||w||_{H^1_x(\mathbb{R})}||u^-_x||_{L^2(\mathbb{R})}+ \frac{( 1 - \tanh^2 y)}{4} \text{ }  | u^+|_\infty||w||_{H^1_x(\mathbb{R})}||u^-_x||_{L^2(\mathbb{R})}$$   
$$ + \frac{( 1 - \tanh^2 y)}{4} \text{ }  | u^-|_\infty||w||_{H^1_x(\mathbb{R})}||u^+_x||_{L^2(\mathbb{R})}.$$
The previous inequality can be written in the form
$$\frac{1}{2} \frac{d}{dt} || w ||^2_{H^1_x(\mathbb{R})} \leq \Big[ ||\partial^{-1}_x \eta_{yy}||_{L^2(\mathbb{R})}+ ( 1 - \tanh^2 y) P\Big(||u^+||_{L^2(\mathbb{R})}, ||u^-||_{L^2(\mathbb{R})}, |u^+|_\infty, | u^-|_\infty \Big) \Big]|| w ||_{H^1_x(\mathbb{R})}$$
$$+\frac{1}{2}\Big[ (1 + \tanh y)||u^+||_{H^1(\mathbb{R})}  + (1 - \tanh y)||u^-||_{H^1(\mathbb{R})}  \Big]|| w ||^2_{H^1_x(\mathbb{R})}.$$
Owing to Lemma \ref{lemma:polya} below, we arrive at the following estimate
$$\frac{1}{2} \frac{d}{dt} || w ||^2_{H^1_x(\mathbb{R})} \leq \Big[ ||\partial^{-1}_x \eta_{yy}||_{L^2(\mathbb{R})}+ ( 1 - \tanh^2 y) Q\Big(||u^+||_{H^1(\mathbb{R})}, ||u^-||_{H^1(\mathbb{R})} \Big) \Big]|| w ||_{H^1_x(\mathbb{R})}$$
$$+\frac{1}{2}\Big[ (1 + \tanh y)||u^+||_{H^1(\mathbb{R})}  + (1 - \tanh y)||u^-||_{H^1(\mathbb{R})}  \Big]|| w ||^2_{H^1_x(\mathbb{R})}$$
or
$$ \frac{1}{2} \frac{d}{dt} || w ||^2_{H^1_x(\mathbb{R})} \leq (D_\eta + C_k)|| w ||_{H^1_x(\mathbb{R})} + (C_1^+ + C_1^-)||w||^2_{H^1_x(\mathbb{R})},$$
where the terms $D_\eta$, $C_k$, $C_1^+$ and $C_1^-$ are defined in Lemma \ref{lemma:Cstarfinite} below.
From this the following inequality is derived
$$ \frac{d}{dt} || w ||_{H^1_x(\mathbb{R})} \leq (D_\eta + C_k) + (C_1^+ + C_1^-)||w||_{H^1_x(\mathbb{R})}.$$
By a variant of Gronwall's lemma, we have
\begin{equation*}
|| w ||_{H^1_x(\mathbb{R})}\leq || w(x,y,0) ||_{H^1_x(\mathbb{R})} e^{(C_1^+ + C_1^-)t} +  \frac{D_\eta +C_k}{C_1^+ + C_1^-}\Big(e^{(C_1^+ + C_1^-)t} -1\Big)
\end{equation*}
or
\begin{equation*}
|| w ||_{H^1_x(\mathbb{R})}\leq || w(x,y,0) ||_{H^1_x(\mathbb{R})} e^{(C_1^+ + C_1^-)t} +  C^\star \Big(e^{(C_1^+ + C_1^-)t} -1\Big).
\end{equation*}

In order to complete the proof we observe the following Lemma.
\begin{lemma}\label{lemma:Cstarfinite}
Let $u^+$ and $u^-$ be nontrivial solutions of the Cauchy problem (\ref{BBM}), i.e. $u^+$ and $u^- \neq 0$ a.e., and $\eta$ be the solution of the Cauchy problem (\ref{BBM-KP}), emanating from Theorems \ref{BBM GWP} and \ref{BBM-KP GWP}, respectively. Consider the following terms:
%---------------------------------------------------------------------------------
\begin{align*} 
	\begin{cases}  
	  	D_\eta :=||\partial^{-1}_x \eta_{yy}||_{L^2(\mathbb{R})} ,\\  
      	C_k :=(1 - \tanh^2 y)Q\Big(||\partial^k_xu^+||_{H^1(\mathbb{R})}, ||\partial^k_xu^-||_{H^1(\mathbb{R})} \Big),\\      
      	C_1^+ :=\frac{1}{2}(1 +\tanh y)||u^+||_{H^1(\mathbb{R})},\\ 
      	C_1^- :=\frac{1}{2}(1 -\tanh y)||u^-||_{H^1(\mathbb{R})},
	\end{cases} 
    \end{align*} 
 where $Q\Big(||\partial^k_xu^+||_{H^1(\mathbb{R})}, ||\partial^k_xu^-||_{H^1(\mathbb{R})} \Big)$ is defined in Lemma \ref{lemma:polya} below. Define the quotient
	\begin{equation*}
	   C^{\star}:=\frac{D_\eta +C_k}{C_1^+ + C_1^-}.
	\end{equation*}
Then, it follows that
	\begin{equation*}
	    \lim _{y\rightarrow \pm \infty} C^{\star}=0.
	\end{equation*}
%-------------------------------------------
\begin{proof}
Let $y\rightarrow \pm \infty$, the assumption that $u^+$ and $u^- \neq 0$ a.e. implies that $||u^+||_{L^2(\mathbb{R})} \wedge ||u^-||_{L^2(\mathbb{R})} \neq 0$, consequently, $||u^+||_{H^k(\mathbb{R})} \wedge ||u^-||_{H^k(\mathbb{R})} \neq 0$ for all $k \geq 0$ and the denominator of $C^{\star}$ is always positive. Concerning the limiting behavior of the numerator, the term $D_\eta =||\partial^{-1}_x \eta_{yy}||_{L^2(\mathbb{R})} \rightarrow 0$ since 
$\partial^{-1}_x \eta_{yy}$ vanishes for $s\geq 2$.
Furthermore, the term 
\begin{equation*}
    C_k =(1 - \tanh^2 y)Q\Big(||\partial^k_xu^+||_{H^1(\mathbb{R})}, ||\partial^k_xu^-||_{H^1(\mathbb{R})} \Big) \rightarrow 0,
\end{equation*}
due to the fact that $(1 - \tanh^2 y)$ vanishes and $Q\Big(||\partial^k_xu^+||_{H^1(\mathbb{R})}, ||\partial^k_xu^-||_{H^1(\mathbb{R})} \Big) < \infty$ due to Theorem \ref{BBM GWP}.
Therefore, the numerator of $C^{\star}$ vanishes, but the denominator does not and the lemma is proven. 
\end{proof}
\end{lemma}

%%%
On account of Lemma \ref{lemma:Cstarfinite}, it follows that
$$\lim _{y\rightarrow \pm \infty} || w ||_{H^1_x(\mathbb{R})}\leq 0 \cdot e^{C_1^\pm t} + 0\cdot\Big(e^{C_1^\pm t} -1\Big)=0.$$
Therefore, we have established that
$$\lim _{y\rightarrow \pm \infty} || w ||_{H^1_x(\mathbb{R})} =\lim_{y \rightarrow \pm \infty} || \eta(x,y,t) - u^{\pm} (x,t)||_{H^1_x(\mathbb{R})}=0.$$
More generally, for $k \geq 1$, we apply the operator $\partial^k_x$ to both sides of the differential equation \ref{wIVP}, multiply the result by $\partial^k_x w$, and integrate over $\mathbb{R}$ in the spatial variable $x$. After a few integration by parts we arrive at the following integral equation
$$\frac{1}{2}\frac{d}{dt} \Big[ \int_{\mathbb{R}}   (\partial^k_xw)^2 \,dx + \int_{\mathbb{R}}  (\partial^{k+1}_xw)^2 \,dx \Big] = \int_{\mathbb{R}} \partial^k_xw \partial^{k-1}_x \eta_{yy}\,dx$$
$$+\frac{(-1)^{k}}{2}(1+ \tanh y) \int_{\mathbb{R}} \Big(\partial^{2k+1}_xw \Big)(u^+w) \,dx + \frac{(-1)^{k}}{2}(1-\tanh y) \int_{\mathbb{R}} \Big(\partial^{2k+1}_xw \Big)(u^-w) \,dx$$
$$+ \frac{(-1)^k}{4}( 1 - \tanh^2 y)\int_{\mathbb{R}} \Big(\partial^{2k}_x w \Big)(u^+ u^+_x)\,dx+ \frac{(-1)^k}{4}( 1 - \tanh^2 y)\int_{\mathbb{R}} \Big(\partial^{2k}_x w \Big)(u^- u^-_x)\,dx$$          
$$ +\frac{(-1)^{k+1}}{4}( 1 - \tanh^2 y)\int_{\mathbb{R}} \Big(\partial^{2k}_x w \Big)(u^+ u^-_x)\,dx+\frac{(-1)^{k+1}}{4}( 1 - \tanh^2 y)\int_{\mathbb{R}} \Big(\partial^{2k}_x w \Big)(u^- u^+_x)\,dx.$$ 
Similarly, the above equation delivers the bound
$$\frac{1}{2} \frac{d}{dt} \Big[|| \partial^k_xw||^2_{L^2(\mathbb{R})} + || \partial^{k+1}_xw||^2_{L^2(\mathbb{R})} \Big ] \leq\int_{\mathbb{R}} |\partial^k_xw \partial^{k-1}_x \eta_{yy} | \,dx  $$
$$+  \frac{( 1 + \tanh y)}{2} |u^+|_\infty  \int_{\mathbb{R}} \Big |\partial^k_x w \partial^{k+1}_xw \Big|\,dx +  \frac{( 1 - \tanh y)}{2}| u^-|_\infty\int_{\mathbb{R}} \Big |\partial^k_x w \partial^{k+1}_xw \Big |\,dx$$
$$ +  \frac{( 1 - \tanh^2 y)}{4} | u^+|_\infty\int_{\mathbb{R}} \Big |\partial^{k+1}_xw    \partial^k_x u^+ \Big |\,dx  +  \frac{( 1 - \tanh^2 y)}{4} | u^-|_\infty \int_{\mathbb{R}} \Big|\partial^{k+1}_xw    \partial^k_x u^- \Big |\,dx$$ 
$$+  \frac{( 1 - \tanh^2 y)}{4}| u^+|_\infty\int_{\mathbb{R}}\Big | \partial^{k+1}_xw   \partial^k_x u^-\Big | \,dx  +   \frac{( 1 - \tanh^2 y)}{4}| u^-|_\infty \int_{\mathbb{R}} \Big|\partial^{k+1}_xw    \partial^k_x u^+\Big| \,dx .$$
An appeal to  H$\ddot{o}$lders inequality results in
$$\frac{1}{2} \frac{d}{dt} \Big[|| \partial^k_xw||^2_{L^2(\mathbb{R})} + || \partial^{k+1}_xw||^2_{L^2(\mathbb{R})} \Big ] \leq ||\partial^k_xw||_{L^2(\mathbb{R})}  ||\partial^{k-1}_x \eta_{yy}||_{L^2(\mathbb{R})}+$$
$$\frac{( 1 + \tanh y)}{2} | u^+|_\infty||\partial^k_xw||_{L^2(\mathbb{R})}||\partial^{k+1}_xw||_{L^2(\mathbb{R})} +  \frac{( 1 - \tanh y)}{2}| u^-|_\infty||\partial^k_xw||_{L^2(\mathbb{R})}||\partial^{k+1}_xw||_{L^2(\mathbb{R})}$$
$$+ \frac{( 1 - \tanh^2 y)}{4} \text{ } | u^+|_\infty||\partial^{k+1}_xw||_{L^2(\mathbb{R})}||\partial^k_xu^+||_{L^2(\mathbb{R})}$$
$$+ \frac{( 1 - \tanh^2 y)}{4}\text{ }  | u^-|_\infty||\partial^{k+1}_xw||_{L^2(\mathbb{R})}||\partial^k_xu^-||_{L^2(\mathbb{R})}$$
$$+ \frac{( 1 - \tanh^2 y)}{4}\text{ }  | u^+|_\infty||\partial^{k+1}_xw||_{L^2(\mathbb{R})}||\partial^k_xu^-||_{L^2(\mathbb{R})}$$
$$ + \frac{( 1 - \tanh^2 y)}{4}\text{ }  | u^-|_\infty||\partial^{k+1}_xw||_{L^2(\mathbb{R})}||\partial^k_xu^+||_{L^2(\mathbb{R})}.$$
After invoking Young's inequality, we have
$$ \frac{1}{2} \frac{d}{dt} \Big[|| \partial^k_xw||^2_{L^2(\mathbb{R})} + || \partial^{k+1}_xw||^2_{L^2(\mathbb{R})} \Big ] \leq \Big[|| \partial^k_xw||^2_{L^2(\mathbb{R})} + || \partial^{k+1}_xw||^2_{L^2(\mathbb{R})} \Big ]^\frac{1}{2}  ||\partial^{k-1}_x \eta_{yy}||_{L^2(\mathbb{R})}$$
$$ +  \frac{( 1 + \tanh y)}{2}| u^+|_\infty\Big[|| \partial^k_xw||^2_{L^2(\mathbb{R})} + || \partial^{k+1}_xw||^2_{L^2(\mathbb{R})} \Big ]$$
$$  + \frac{( 1 - \tanh y)}{2}| u^-|_\infty\Big[|| \partial^k_xw||^2_{L^2(\mathbb{R})} + || \partial^{k+1}_xw||^2_{L^2(\mathbb{R})} \Big ]$$     
$$ + \frac{( 1 - \tanh^2 y)}{4} \text{ } | u^+|_\infty\Big[|| \partial^k_xw||^2_{L^2(\mathbb{R})} + || \partial^{k+1}_xw||^2_{L^2(\mathbb{R})} \Big ]^\frac{1}{2}||\partial^k_xu^+||_{L^2(\mathbb{R})}$$       
$$ + \frac{( 1 - \tanh^2 y)}{4}\text{ }  | u^-|_\infty\Big[|| \partial^k_xw||^2_{L^2(\mathbb{R})} + || \partial^{k+1}_xw||^2_{L^2(\mathbb{R})} \Big ]^\frac{1}{2}||\partial^k_xu^-||_{L^2(\mathbb{R})}$$
$$+ \frac{( 1 - \tanh^2 y)}{4}\text{ }  | u^+|_\infty\Big[|| \partial^k_xw||^2_{L^2(\mathbb{R})} + || \partial^{k+1}_xw||^2_{L^2(\mathbb{R})} \Big ]^\frac{1}{2}||\partial^k_xu^-||_{L^2(\mathbb{R})}$$   
$$ + \frac{( 1 - \tanh^2 y)}{4}\text{ }  | u^-|_\infty\Big[|| \partial^k_xw||^2_{L^2(\mathbb{R})} + || \partial^{k+1}_xw||^2_{L^2(\mathbb{R})} \Big ]^\frac{1}{2}||\partial^k_xu^+||_{L^2(\mathbb{R})}.$$

To obtain the required bounds we invoke the following lemma. 
%%%
\begin{lemma}\label{lemma:polya}
Suppose that $k \geq 0$ and define the following polynomials:
%%%%
%---------------------------------------------------------------------------------
\begin{align*} 
	  	P\Big(||\partial^k_xu^+||_{L^2(\mathbb{R})}, ||\partial^k_xu^-||_{L^2(\mathbb{R})}, |u^+|_\infty, | u^-|_\infty \Big) &:=\frac{1}{4}\Big(| u^+|_\infty ||\partial^k_xu^+||_{L^2(\mathbb{R})}+| u^-|_\infty ||\partial^k_xu^-||_{L^2(\mathbb{R})} \\  &+ | u^+|_\infty ||\partial^k_xu^-||_{L^2(\mathbb{R})}+| u^-|_\infty ||\partial^k_xu^+||_{L^2(\mathbb{R})} \Big),\\
      Q\Big(||\partial^k_xu^+||_{H^1(\mathbb{R})}, ||\partial^k_xu^-||_{H^1(\mathbb{R})} \Big) &:=\frac{1}{4}\Big(||\partial^k_xu^+||^2_{H^1(\mathbb{R})}+||\partial^k_xu^-||^2_{H^1(\mathbb{R})} \\
      &+ ||\partial^k_xu^+||_{H^1(\mathbb{R})} ||\partial^k_xu^-||_{H^1(\mathbb{R})} \\
      &+||\partial^k_xu^-||_{H^1(\mathbb{R})} ||\partial^k_xu^+||_{H^1(\mathbb{R})} \Big),\\
    \end{align*} 
where the operator $\partial^0_x$ is regarded as the identity operator. Then, the following estimate holds
\begin{equation}
P\Big(||\partial^k_xu^+||_{L^2(\mathbb{R})}, ||\partial^k_xu^-||_{L^2(\mathbb{R})}, |u^+|_\infty, | u^-|_\infty \Big) \leq Q\Big(||\partial^k_xu^+||_{H^1(\mathbb{R})}, ||\partial^k_xu^-||_{H^1(\mathbb{R})} \Big).    
\label{polyestimate}
\end{equation}
\end{lemma}
\begin{proof}
To establish the necessary estimate, we focus on specific terms appearing in the explicit expressions of each polynomial. Clearly, the terms of $P$ involving $L^2(\mathbb{R})$ norms of the functions $\partial^k_xu^\pm$ can all be bounded above by their corresponding norms in $H^1(\mathbb{R})$. In order to bound the $L^\infty(\mathbb{R})$ norms of the functions $u^\pm$ in the resulting terms, we utilize the well known Sobolev embedding $H^s(\mathbb{R}^n) \hookrightarrow C^r(\mathbb{R}^n)$, provided $s-r >\frac{n}{2}$. Particularly, the case of $r=0$ corresponding to the embedding $H^s(\mathbb{R}^n) \hookrightarrow L^\infty(\mathbb{R}^n)$ is invoked to obtain the required bounds. Therefore, the inequality given by Equation \ref{polyestimate} in the statement of the lemma is confirmed for all $k \geq 0$.
\end{proof}

On account of Lemma \ref{lemma:polya}, we deduce that 
$$\frac{1}{2} \frac{d}{dt} || w ||_{H^k_x(\mathbb{R})}^2 \leq \Big[ ||\partial^{k-1}_x \eta_{yy}||_{L^2(\mathbb{R})}+( 1 - \tanh^2 y) Q\Big(||\partial^k_xu^+||_{H^1(\mathbb{R})}, ||\partial^k_xu^-||_{H^1(\mathbb{R})} \Big)\Big]|| w ||_{H^k_x(\mathbb{R})}$$
$$+\frac{1}{2}\Big[ (1 + \tanh y)||u^+||_{H^1(\mathbb{R})}+ ( 1 - \tanh y)||u^-||_{H^1(\mathbb{R})}\Big]|| w ||_{H^k_x(\mathbb{R})}^2.$$
This leads to the following differential inequality
$$ \frac{d}{dt} || w ||_{H^k_x(\mathbb{R})} \leq (D_\eta + C_k) + (C_1^+ + C_1^-)||w||_{H^k_x(\mathbb{R})}.$$
Proceeding with a variant of Gronwall's lemma in combination with Lemma \ref{lemma:Cstarfinite}, it follows that
 \begin{equation*}
 || w ||_{H^k_x(\mathbb{R})}\leq || w(x,y,0) ||_{H^k_x(\mathbb{R})} e^{(C_1^+ + C_1^-)t} +  C^\star\Big(e^{(C_1^+ + C_1^-)t} -1\Big).
 \end{equation*}
Placing this together, we conclude that
$$\lim _{y\rightarrow \pm \infty} || w ||_{H^k_x(\mathbb{R})} =\lim_{y \rightarrow \pm \infty} || \eta(x,y,t) - u^{\pm} (x,t)||_{H^k_x(\mathbb{R})}=0,$$
which proves the theorem. 
\end{proof}

\newpage 

\begin{remark}
The relationship between the regularity parameters $k$ and $s$ is primarily due to the presence of the term $\partial^{k-1}_x \eta_{yy}$.  For purposes of exposition this relationship is summarized in Table 1 below.

{\footnotesize
\begin{table}[h!]
\label{table:nonlin}
\begin{threeparttable}
\caption{Regularity Parameters $k$ and $s$}
\tabcolsep=0.5cm
\begin{tabular}{lll}
\hline
\textbf{$k$} &
\textbf{$\partial^{k-1}_x \eta_{yy}$} &
\textbf{$s$}\\
\hline
$0$ & $\partial^{-1}_x \eta_{yy}$ & $s \geq 1$\\
$1$ & $\eta_{yy}$ & $s \geq 1$\\
$2$ & $\partial_x \eta_{yy}$ & $s \geq 3$ \\
$3$ & $\partial^{2}_x \eta_{yy}$ & $s \geq 4$\\
\vdots & \vdots & \vdots \\
$k$ & $\partial^{k-1}_x \eta_{yy}$ & $s \geq k-1+2=k+1$\\
 \hline
%&  \textbf{Mean} & $x$ & $y$ \\
%& \textbf{Variance} & $x$ & $y$\\
%\hline
\end{tabular}
 \end{threeparttable}
  \end{table}
  \noindent This table depicts the relationship between the regularity parameters $k$ and $s$. The inequality $s \geq k+1$
guarantees sufficient regularity to ensure that $\partial^{k-1}_x \eta_{yy} \rightarrow 0$ as the transverse variable $y \rightarrow \pm \infty$.
}
\end{remark}
% ===============================================
\section*{Acknowledgements} 
 J.A. is grateful for the support from the Department of Mathematics at Louisiana State University and A\&M College, the University of Georgia at Athens and Saint Leo University. M.T. is grateful for the working environment and support from the Department of Mathematics at Louisiana State University and A\&M College. 
 % ===============================================
\bibliographystyle{siam}
\bibliography{bio} 
\end{document}